\documentclass[12pt,a4paper,reqno]{amsart} 
\usepackage{amssymb}
\usepackage{latexsym}
\usepackage{amsmath}
\usepackage{mathrsfs}
\usepackage[X2,T1]{fontenc}
\usepackage[applemac]{inputenc}
\usepackage{cite}
\usepackage{calc}                   

\newcommand{\scal}[2]{\langle #1,#2\rangle}
\newcommand{\rr}[1]{\mathbf R^{#1}}

\newcommand{\nm}[2]{\Vert #1\Vert _{#2}}

\newcommand{\op}{\operatorname{Op}}

\newcommand{\sets}[2]{\{ \, #1\, ;\, #2\, \} }

\newcommand{\fy}{\varphi}
\newcommand{\cdo}{\, \cdot \, }

\newcommand{\wpr}{{\text{\footnotesize $\#$}}}

\newcommand{\eabs}[1]{\langle #1\rangle}     

\newcommand{\vrum}{\vspace{0.1cm}}

\newcommand{\nn}[1]{{\mathbf N}^{#1}}
\newcommand{\maclS}{\mathcal S}

\def\sS{{\mathscr S}}
\def\cS{{\mathcal S}}

\def\R{\mathbf{R}}
\def\N{\mathbf{N}}

\numberwithin{equation}{section}          

\newtheorem{thm}{Theorem}
\numberwithin{thm}{section}
\newtheorem*{tom}{\rubrik}
\newcommand{\rubrik}{}
\newtheorem{prop}[thm]{Proposition}
\newtheorem{cor}[thm]{Corollary}
\newtheorem{lemma}[thm]{Lemma}

\theoremstyle{definition}

\theoremstyle{remark}
\newtheorem{rem}[thm]{Remark}

\author{Marco Cappiello}

\address{Dipartimento di Matematica ``G. Peano", Universit\`a di Torino,
Via Carlo Alberto 10, Torino, Italy}

\email{marco.cappiello@unito.it}

\author{Luigi Rodino}

\address{Dipartimento di Matematica ``G. Peano", Universit\`a di Torino,
Via Carlo Alberto 10, Torino, Italy}

\email{luigi.rodino@unito.it}

\author{Joachim Toft}

\address{Department of Mathematics,
Linn{\ae}us University, V{\"a}xj{\"o}, Sweden}

\email{joachim.toft@lnu.se}

\title{On the inverse to the harmonic oscillator}



\begin{document}

\begin{abstract}
Let $b_d$ be the Weyl symbol of the inverse to the
harmonic oscillator on $\R^d$. We prove that $b_d$ and its derivatives
satisfy convenient bounds of Gevrey and Gelfand-Shilov
type, and obtain explicit expressions
for $b_d$. In the even-dimensional case we characterize $b_d$ in terms
of elementary functions.

\par

In the analysis we use properties of radial symmetry and a combination of different 
techniques involving classical a priori estimates, commutator identities,
power series and asymptotic expansions.
\end{abstract}

\maketitle

\par

\section{Introduction}\label{sec0}

\par

A fundamental operator in quantum physics and classical analysis
is the harmonic oscillator
\begin{equation}\label{harmonicop}
H= H_d=-\Delta +|x|^{2}, \qquad x \in \R^{d}.
\end{equation} 
In physics the operator $H$ appears in the stationary Schr\"odinger equation
for a particle under the action of a quadratic potential. In classical
analysis, $H$ is also known as the Hermite operator, and possesses several
convenient properties. For example,
the operator $H$
is strictly positive in $L^2(\R^d)$ with discrete spectrum, and the eigenfunctions 
are the Hermite functions, see for example \cite{E, RS, Wo}.
\par
By means of the Hermite functions one can express also the kernel of the inverse $H^{-1}.$ 
On the other hand, coherently with the point of
view of the quantum physics, $H^{-1}$ can be written as Weyl pseudo-differential operator
\begin{equation} \label{H^-1}
H^{-1}f(x)= (2\pi)^{-d} \iint e^{i \langle x-y,\xi \rangle}
b\left ((x+y)/2, \xi \right) f(y)\, dy d\xi,
\end{equation}
for a suitable symbol $b(x,\xi)= b_d(x,\xi)$ in $\R^{2d}$ (see for example the
general calculus 
in \cite{He, Ho1, NiRo, Sh} for classes of symbols and operators in $\R^{d}$).
The calculus provides, as a particular case, the construction in these classes of
the symbol of the parametrix of $H.$ Also the symbol of the inverse $b_d(x,\xi)$ in 
\eqref{H^-1} belongs to the same classes, in  view of the property of spectral invariance 
(cf. \cite{BC}).
Despite the power of the pseudo-differential theory, the study of the peculiar properties
of $b_d(x,\xi)$ is missing in literature.

\par  

The aim of the paper is to analyze the function $b_d(x,\xi)$ in $\R^{2d}$ and to
derive suitable regularity estimates and explicit expressions. In the
even-dimensional case we express $b_d$ in terms of elementary
functions. Just to have a sample of our study, we here mention the striking and seemingly 
unnoticed formula in dimension $d=2$:
\begin{equation} \label{bdim2}
b_2(x,\xi)= \frac{1-e^{-|x|^2-|\xi|^2}}{|x|^2+|\xi|^2}.
\end{equation}
Before giving a more detailed presentation of our results, we recall
some known facts for $H$.
First, we have convenient bijectivity properties of $H$ in different function,
distribution and ultradistribution spaces. (See e.{\,}g.
\cite{He, Ho1, La, Pi, NiRo, Sh}, and Proposition 2.2 and
Theorem 3.10 in \cite{SiTo}.)
Furthermore, the operator possesses useful regularity
properties. For example, if $f\in \sS '(\rr d)$, and 
$$
Hf \in \sS(\rr d) \quad \text{or}\quad H^Nf\in L^2(\rr d)
$$
for every $N$, then $f\in \sS (\rr d)$. 
One way to obtain the latter property is
to use Theorem 3.10 in \cite{SiTo}. The other standard way is to use Theorem
25.4 (with $m=m_0=2$) in \cite{Sh}, which implies that 
\begin{equation}\label{hoscest1}
|\partial ^{\alpha} b_d(x,\xi )|\le C_\alpha \eabs {(x,\xi )}^{-2-|\alpha |},
\end{equation}
and using appropriate mapping properties of pseudo-differential operators
with symbols satisfying \eqref{hoscest1}.
In \cite{CGR1} more refined estimates are established for the solutions
$f \in \sS '(\rr d)$ of the equation $Hf=g$ and for more general differential
operators when $g$ belongs to the Gelfand-Shilov space $\cS _{s}(\R ^d)$,
$s \ge 1/2$ (cf. Section \ref{sec1} for the definitions).

\par

The first aim of this paper is to establish certain
refinements as well as other estimates related to \eqref{hoscest1}.
Imitating the local analytic calculus of \cite{BMK} and trying to
adapt the global Gelfand-Shilov calculus of \cite{CR} one could
tentatively assume
\begin{equation}\label{estimateh}
|\partial ^{\alpha} b_d(x,\xi )|\le C^{|\alpha|+1}(\alpha!)^{s_0}
\eabs {(x,\xi )}^{-2-|\alpha |},
\end{equation}
for some $s_0 \geq 1/2$ and positive constant $C$ independent of $\alpha.$ 
A global calculus  for symbols satisfying factorial estimates
of the form \eqref{estimateh} does not exist in the literature, especially for
$1/2 \leq s_0 < 1$, and
sharp estimates for $b_d$ are considered as an open and difficult
problem. Nevertheless in the present paper we prove that the
estimate
\begin{equation}
\label{finalestimatehoGen}
|\partial ^{\alpha}b_d(x,\xi)| \le
C^{|\alpha|+1}(\alpha!)^{(s+1)/2}\eabs {(x,\xi )}^{-2-s|\alpha |},
\end{equation}
holds for some positive constant $C$ which is independent of
$\alpha \in \nn {2d}$ and $s\in [0,1]$. In particular, for $s=0$ we have
\begin{equation}
\label{finalestimateho}
|\partial ^{\alpha}b_d(x,\xi)| \le
C^{|\alpha|+1}(\alpha!)^{1/2}\eabs {(x,\xi )}^{-2},
\end{equation}
whereas, for $s=1$, \eqref{hoscest1} is refined into
\begin{equation}
\label{finalestimateho2}
|\partial ^{\alpha}b_d(x,\xi)| \le
C^{|\alpha| +1}\alpha !\eabs {(x,\xi )}^{-2-|\alpha |},
\end{equation}
for some constant $C$ which is independent of $\alpha \in
\nn {2d}$. Furthermore, we use \eqref{finalestimatehoGen}
to establish similar estimates for $b_{d,t}$, the $t$-symbol of
$H^{-1}$ (cf. Theorem \ref{harmonicest}, Proposition
\ref{harmonicest1}$'$ and Remarks \ref{Remtsymbol} and
\ref{RemharmonicestExt}).

\par

Starting from the estimate \eqref{finalestimatehoGen} for $b_d$, it 
might be interesting to study general symbols satisfying estimates of 
the same type and to establish regularity results for the related operators
in the setting of Gelfand-Shilov spaces as it has been done in \cite{CGR1, CGR2, CGR3}.
We will treat these applications in future papers and focus here only on the 
model $H^{-1}$.

\par

To prove \eqref{finalestimatehoGen}, in Section \ref{sec2}
we use classical a priori estimates for globally elliptic operators, and suitable 
commutator estimates. Moreover we apply the symbolic calculus
to prove that $b_d$ satisfies
\begin{equation} \label{bXequation}
(H_{0}b_d)(X)\equiv |X|^{2}b_d(X) -\frac{1}{4}\Delta_{X}b_d(X)=1,\quad
X=(x,\xi )\in \rr {2d}.
\end{equation}
We note that the operator on the left-hand side is a (dilated) harmonic
oscillator on the phase space variables $X=(x,\xi )\in \rr {2d}$. In 
Sections \ref{sec1} and \ref{sec2} we show that any solution of
\eqref{bXequation} is of the form
\begin{equation}\label{cddef}
b_d(X)=c_d(|X|^2),
\end{equation}
for some entire function $c_d$ on $\mathbf C$. In particular, $b_d$ is
radial symmetric.

\medspace

This introduces the next main issue, which concerns explicit formulas
for $b_d$, and is presented in Section \ref{sec3}. In fact, by the radial
symmetry we may reduce 
\eqref{bXequation} into an ordinary differential equation on $c_d$.
A power series expansion and \eqref{finalestimateho} then give
\begin{multline}\label{c-expression-intro}
c_d(t)
\\[1ex]
= 
\frac{d!!}{d}\left (
\alpha \sum _{p=0}^\infty \frac {t^{2p}}{(2p)!! (2p+d-1)!!}
-  \sum _{p=0}^\infty \frac {t^{2p+1}}{(2p+1)!! (2p+d)!!}
\right ),
\end{multline}
where $\alpha$ is equal to $1$ when $d$ is even, and equal
to $\pi /2$ when $d$ is odd.

\par

The expansion \eqref{c-expression-intro} gives finite analytic
expressions for $c_d$, thereby for $b_d$, when $d=2n$ is even.
Namely, we first obtain an explicit asymptotic expansion
of the symbol $b_d$ in terms of homogeneous functions in $|(x,\xi)|$. 
Then, by a slight modification of this expansion
inspired by \eqref{bdim2}, we establish the general formula 
\begin{equation*}
b_{2n}(x,\xi) = \sum _{j=0}^{n-1} {{n-1}\choose {j}}(-1)^j(2j)!
\frac {1-e^{-|x|^2-|\xi|^{2}}p_{2j}(|x|^2+|\xi|^{2})}{(|x|^2+|\xi |^{2})^{2j+1}},
\end{equation*}
for the symbol $b_{2n}(x,\xi)$, where $p_{2j}(t)$ denotes the Taylor polynomial
of $e^{t}$ of order $2j$ centered at $t=0$ (cf. formula \eqref{bExact} in Section
\ref{sec3}).

\par

For the odd dimensional case, the formula \eqref{c-expression-intro} does
not give any simple expressions of finite numbers of elementary functions.
In fact, if $d=2n+1$, then the first series in \eqref{c-expression-intro}
is equal to $u_n(t^2/4)$, for some Bessel function $u_n$
(cf. Theorem \ref{cdrem}), and for the second series it seems to be even more
complicated to find well-known special functions, since each coefficient contains two factors of odd
semi-factorials. For example, in contrast to the the first series, the second
one can not be completely described by Bessel functions. On the other hand,
by the link
between $c_d$ and $b_d$, a combination of \eqref{finalestimatehoGen} and
\eqref{c-expression-intro} leads to
$$
\sum _{p=0}^\infty \frac {t^{2p+1}}{(2p+1)!!(2p+d)!!} = \frac \pi 2u_n(t^2/4) +
\mathcal O (1/t), \qquad t \to +\infty,
$$
which seems to be unknown until now and should be interesting in the theory
of special functions. (See Theorem \ref{cdrem} and Remark \ref{cdremrem}
for more detailed explanations,
which also include more detailed estimates for the involved
functions and their derivatives).

\par

\section{Preliminaries}\label{sec1}

\par

In this section we recall some basic results on pseudo-differential calculus.
We shall often formulate these results in the
framework of the Gelfand-Shilov space $\cS _{1/2}(\rr d)$ and its dual $\cS _{1/2}'(\rr d)$
(see e.{\,}g. \cite{GS}). The reader who is not interested in this general situation
may replace $\cS _{1/2}(\rr d)$ and $\cS _{1/2}'(\rr d)$ by $\sS (\rr d)$ and $\sS '(\rr d)$
respectively. Here $\sS (\rr d)$ is the set of Schwartz functions on $\rr d$, and
$\sS '(\rr d)$ is the set of tempered distributions on $\rr d$, see 
e.{\,}g. \cite{Ho1}.

\par

We start by recalling the definition of Gelfand-Shilov spaces.
Let $s\ge 1/2$ be fixed. For any $f\in C^\infty (\rr d)$ and $h>0$
we let
\begin{equation*}
\nm f{\mathcal S_{s,h}}\equiv \sup \frac {|x^\beta \partial ^\alpha
f(x)|}{h^{|\alpha | + |\beta |}(\alpha !\, \beta !)^s}.
\end{equation*}
Here the supremum should be taken over all $x\in \rr d$
and multi-indices $\alpha ,\beta \in \mathbf N^d$.
Then the Gelfand-Shilov space
$\cS _s(\rr d)$ consists of all $f\in C^\infty (\rr d)$ such that
$\nm f{\mathcal S_{s,h}}$ is finite for some $h>0$. Evidently,
$\cS _s(\rr d)\subset \sS (\rr d)$ for every $s \geq 1/2$.

\par

The set $\mathcal S_{s}(\rr d)$ contains all finite linear combinations of
Hermite functions. Since such linear combinations are dense in $\mathscr S (\rr d)$,
it follows that the dual $\mathcal S_{s}'(\rr d)$ of $\mathcal S_{s}(\rr d)$ is
a space which contains $\mathscr S'(\rr d)$.

\par

We refer to \cite{GS,NiRo} for more facts about Gelfand-Shilov spaces and their
duals.

\par

Next we recall certain properties of pseudo-differential operators.
Let $t\in \mathbf R$. For any $a\in \cS _{1/2}(\rr {2d})$, the pseudo-differential operator
$\op _t(a)$ is the linear and continuous operator on $\cS _{1/2}(\rr d)$, defined by
\begin{equation}\label{pseudo}
\op _t(a)f(x) = (2\pi )^{-d/2}\iint _{\rr {2d}}a((1-t)x+ty,\xi )f(y) e^{i\scal {x-y}\xi}\,
dyd\xi .
\end{equation}
The definition extends uniquely to any $a\in \cS _{1/2}'(\rr {2d})$, and then
$\op _t(a)$ is continuous from $\cS _{1/2} (\rr d)$ to $\cS _{1/2} '(\rr d)$.
(Cf. \cite {To11}.) In the case
$t=0$, then $\op _0(a)$ agrees with the Kohn-Nirenberg representation
$a(x,D)$, and if $t=1/2$, then $\op _{1/2}(a)$ is equal to the Weyl quantization
$\op ^w(a)$.

\par

Now we recall the definition of the Shubin class of pseudo-differential operators.
Let $m \in \R$. Then the Shubin class $\Gamma^m(\R^{2d})$ is the set of all functions
$a(x,\xi) \in C^{\infty}(\R ^{2d})$ satisfying the estimate
$$
|\partial^{\alpha} a(x,\xi)| \le C_{\alpha}\langle (x,\xi)\rangle^{m-|\alpha|},
\qquad (x,\xi) \in \R^{2d}.
$$
In particular, for the symbols of the harmonic oscillator and its inverse, we have
$h \in \Gamma^{2}(\R^{2d})$ and $b_d \in \Gamma^{-2}(\R^{2d})$.
By (23.17) in \cite{Sh}, 
%
%
%
%
the operators $\op_t(a)$ with
$a \in \Gamma^m(\R^{2d}) $ are continuous on $\sS (\rr d)$,
and on $\sS '(\rr d)$.

\par

A symbol $a\in \Gamma ^m(\rr {2d})$ is said to be globally elliptic if
\begin{equation}\label{globell}
|a(x,\xi )|\ge c|(x,\xi )|^m,\quad \text{when}\quad |(x,\xi )|\ge R,
\end{equation}
for some positive constants $c$ and $R$. 

\par

In the following we  shall prove a result on the radial symmetry of solutions
of the problem
\begin{equation}\label{problem}
\op _t(a)f =g
\end{equation}
where $\op _t(a)$ is the pseudo-differential operator given by
\eqref{pseudo}.  Here we recall that an element $f\in \cS '_{1/2}(\rr d)$
is called \emph{radial symmetric}, if the
pullback $U^*f$ is equal to $f$, for every unitary transformation $U$ on $\rr d$.
In the case when $f$ in addition is a measurable function,
then $f$ is radial symmetric, if and only if $f(x)=f_0(|x|)$ a.{\,}e., for some
measurable function $f_0$ on $\mathbf R$.
%
%

\par

\begin{prop}\label{RadialSymSolutions}
Let $a\in \maclS _{1/2}'(\rr {2d})$ and $f,g\in \maclS _{1/2}'(\rr d)$
be such that $\op _t(a)f$ is well defined and is equal to $g$, and that
$a(x,\xi ) $ is a radial symmetric symbol in the $x$-variable and
in the $\xi$-variable. Then the following is true:
\begin{enumerate}
\item If $f$ is radial symmetric, then $g$ is radial symmetric;

\vrum

\item If $\op _t(a)$ is injective and $g$ is radial symmetric, then $f$
is radial symmetric.
\end{enumerate}
\end{prop}
%
%
%

\par

\begin{proof}
Assume that $f$ is radial symmetric and let $U$ be a unitary matrix on $\rr d$.
Then formal computations give
\begin{multline*}
g(Ux) = (2\pi )^{-d}\iint a((1-t)Ux+ty,\xi  )f(y)e^{i\scal {Ux -y}\xi}\, dyd\xi 
\\[1ex]
=(2\pi )^{-d}\iint a((1-t)Ux+ty,U\xi )f(y)e^{i\scal {x -U^{-1}y}\xi}\, dyd\xi 
\\[1ex]
=(2\pi )^{-d}\iint a(U((1-t)x+ty),\xi  )f(y)e^{i\scal {x -y}\xi}\, dyd\xi 
\\[1ex]
=(2\pi )^{-d}\iint a((1-t)x+ty,\xi )f(y)e^{i\scal {x -y}\xi}\, dyd\xi 
= g(x),
\end{multline*}
which proves that $g$ is radial symmetric. Hence (1) holds.

\par

Assume instead that $g$ is radial symmetric and that $\op _t(a)$ is
injective. Again let $U$ be an arbitrary unitary matrix. By
\eqref{problem} we have
\begin{multline*}
g(x) = g(Ux)
= (2\pi )^{-d}\iint a((1-t)Ux+ty,\xi  )f(y)e^{i\scal {Ux-y}\xi}\, dyd\xi 
\\[1ex]
= (2\pi )^{-d}\iint a(U((1-t)x+ty),U\xi )f(Uy)e^{i\scal {U(x-y)}{U\xi}}\, dyd\xi 
\\[1ex]
= (2\pi )^{-d}\iint a((1-t)x+ty,\xi )f(Uy)e^{i\scal {x-y}{\xi}}\, dyd\xi .
\end{multline*}
Hence both $f$ and $U^*f$ solves \eqref{problem}. Since
$\op (a)$ is injective, it follows that $f=U^*f$. Consequently,
$f$ is radial symmetric, and (2) follows. The proof is complete.
\end{proof}

\par

Proposition \ref{RadialSymSolutions} applies in particular to the
harmonic oscillator giving the following result.

\par

\begin{cor}\label{harmonicradialsymmetric}
Let $f,g\in \cS '_{1/2}(\rr d)$ be such that 
\begin{equation}\label{harmoniceq}
(-\Delta +C|x|^2)f=g
\end{equation}
for some constant $C>0$. Then $f$ is radial symmetric if
and only if $g$ is radial symmetric.
\end{cor}

\par

\begin{rem}\label{RemGenHarmOsc}
In the literature it is common to add a constant (the spectral parameter)
to the harmonic oscillator and to consider the more general equation
\begin{equation}\tag*{(\ref{harmoniceq})$'$}
(-\Delta +C_1|x|^2+C_2)f = g.
\end{equation}
For example, the Helmholz equation is of this form.

\par

In particular, Corollary \ref{harmonicradialsymmetric} can be extended into the following.

\medspace

\noindent
\emph{Let $f,g\in \cS '_{1/2}(\rr d)$ be such that \eqref{harmoniceq}$'$ holds
for some constants
\begin{equation}\label{C1C2consts}
C_1>0\quad \text{and}\quad C_2\in \mathbf C\setminus
\sets {C_1^{1/2}(-d-2n)}{n\in \mathbf N}.
\end{equation}
Then $f$ is radial symmetric if and only if $g$ is radial symmetric.}
\end{rem}
%

\section{Estimates for the inverse of the 
harmonic oscillator}\label{sec2}

\par

In this section we derive estimates for the Weyl symbol of the inverse to the
harmonic oscillator. In the last part of the section we shall use these results
to obtain related estimates for the $t$-symbol of that inverse.
More precisely, we prove the following result, which in the
case $s=1$ gives more
detailed information about $b_d$ compared to the Shubin estimate
\eqref{hoscest1}.

\par

\begin{thm}\label{harmonicest}
Let $b_d$ be the Weyl symbol of the inverse to the
harmonic oscillator on $\rr d$. Then there is
a constant $C>0$ such that \eqref{finalestimatehoGen}
holds for every $\alpha \in \N^{2d}$ and $s\in [0,1]$.
\end{thm}

\par

By an argument with geometric mean-values, it suffices to prove the
result in the limit cases $s=0$ and $s=1$, which 
correspond to the estimates \eqref{finalestimateho}
and \eqref{finalestimateho2},  respectively, Since these cases are interesting
by their own we write them as two independent statements.

\par 

\begin{prop}\label{harmonicest1}
Let $b_d$ be the Weyl symbol of the inverse to the
harmonic oscillator on $\rr d$. Then there is a constant $C>0$ such
that \eqref{finalestimateho} holds for every $\alpha \in \N^{2d}$.
\end{prop}

\par

\begin{prop}\label{harmonicest2}
Let $b_d$ be the Weyl symbol of the inverse to
the harmonic oscillator on $\rr d$. Then there is a constant $C>0$
such that \eqref{finalestimateho2} holds for every $\alpha \in \N^{2d}$.
\end{prop}

\par

In order to prove Proposition \ref{harmonicest1} we need some preparation.
The invertibility properties and the symbolic calculus give
\begin{equation}\label{bdWeylProd}
(|x|^{2}+|\xi|^{2}) \wpr b_d(x,\xi)=1,
\end{equation}
where $\wpr$  is the Weyl product (cf. Section 18.5 in \cite{Ho1}). We claim that
\begin{equation}\label{bequation}
H_{0}b_d(x,\xi)\equiv (|x|^{2}+|\xi|^{2})b_d(x,\xi) -\frac{1}{4}\Delta_{x,\xi}b_d(x,\xi)=1
\end{equation}
and 
\begin{equation}\label{im}
\sum_{j=1}^{d}(x_{j}\partial_{\xi_{j}}b_d(x,\xi)-\xi_{j}\partial_{x_{j}}b_d(x,\xi))=0.
\end{equation}

\par

In fact, by \eqref{bdWeylProd} and asymptotic expansion we get
\begin{equation}
\label{weylproduct}
H_{0}b_d(x,\xi) + i \sum_{j=1}^{d}(x_{j}\partial_{\xi_{j}}b_d(x,\xi)-
\xi_{j}\partial_{x_{j}}b_d(x,\xi)) = 1.
\end{equation}
Since $H$ is self-adjoint, it follows that $\op ^w(b_d)$ is also self-adjoint.
By using the fact that a Weyl operator is self-adjoint, if and only if
its Weyl symbol is real-valued, it follows that $b_d$ is real-valued. 
Hence \eqref{weylproduct} gives \eqref{bequation} and \eqref{im}.

\par

\begin{proof}[Proof of Proposition \ref{harmonicest1}]
For $\alpha=0$ the assertion is true since $b_d$ is in $\Gamma^{-2}(\R^{2d})$.
Assume instead that $\alpha \neq 0$.
By letting $X=(x,\xi)$, it follows that \eqref{bequation} is the same as 
\eqref{bXequation}. Since $H_{0}$ is globally elliptic,
we have
\begin{equation}
\label{ellipticestimate}
\sum_{|\gamma+\delta| \leq 2}\| X^{\gamma}\partial_{X}^{\delta}u
\|_{L^{p}(\R^{2d})} \leq C_p \| H_{0}u \|_{L^{p}(\R^{2d})}
\end{equation}
for every $p \in (1,\infty)$, $u \in \sS(\R^{2d})$ and for some constant
$C_p$ depending on $p$ and $d$ only (cf. \cite{NiRo}). From now
on, let $p>2d$. With this choice, $b_d(X) \in L^{p}(\R ^{2d})$
together with all its derivatives and the same holds for $\langle  X \rangle b_d(X),$
since $b_d \in \Gamma^{-2}(\R^{2d})$.
Now let $u = \partial^{\alpha}_{X}b_d$ in \eqref{ellipticestimate}, where
$\alpha \in \N^{2d} \setminus \{0\}$. In order to obtain appropriate
estimates  we consider the commutator
\begin{equation}\label{commutator}
[H_{0}, \partial_{X}^{\alpha}]b_d= H_{0}(\partial_{X}^{\alpha}b_d) -
\partial_{X}^{\alpha}(H_{0}b _d) = H_{0}(\partial_{X}^{\alpha}b_d),
\end{equation}
since $H_{0}b_d=1.$ By combining \eqref{ellipticestimate} and
\eqref{commutator} we get
\begin{equation}\label{triangle2}
\sum_{|\gamma+\delta| \leq 2}\| X^{\gamma}\partial_{X}^{\delta+\alpha}b_d
\|_{L^{p}(\R^{2d})} \leq C \| [H_{0}, \partial_{X}^{\alpha}]b_d \|
_{L^{p}(\R^{2d})}.
\end{equation}
Since $[\Delta _X,\partial _X^\alpha ] =0$, we have
\begin{multline}\label{triangle3}
[H_{0}, \partial_{X}^{\alpha}]b_d =
[|X|^{2}, \partial_{X}^{\alpha}]b_d
\\[1ex]
= -2\sum _{\stackrel{1\leq j\leq 2d}{\alpha
_{j}\neq 0}}\alpha_{j}X_{j}\partial_{X}^{\alpha-e_{j}}b_d -
\sum_{\stackrel{1\leq j \leq 2d}{\alpha_{j} \geq 2}} \alpha_{j}
(\alpha_{j}-1)\partial_{X}^{\alpha-2e_{j}}b_d,
\end{multline} 
where $e_j$, $j=1,\dots ,2d$, is the standard basis in $\rr {2d}$.

\par

Now we set, for $\alpha \neq 0$:
$$
J_{\alpha}= \sum_{\stackrel{|\gamma+\delta|\le 2}{(\gamma,\delta)\ne (0,0)}}
\| X^{\gamma}\partial_{X}^{\delta + \alpha}b_d\|_{L^{p}(\R^{2d})}. 
$$
From \eqref{triangle2} and \eqref{triangle3} it follows that 
\begin{multline}\label{reduction}
J_{\alpha}  \leq 2C\sum_{\stackrel{1\leq j \leq 2d}{\alpha_{j}\neq 0}}
\alpha_{j}\|X_{j}\partial_{X}^{\alpha-e_{j}}b_d \|_{L^{p}(\R ^{2d})}
+ C\sum_{\stackrel{1\leq j \leq 2d}{\alpha_{j} \geq 2}} \alpha_{j}(\alpha_{j}-1)\|
\partial_{X}^{\alpha-2e_{j}}b_d \|_{L^{p}(\R ^{2d})}
\\[1ex]
=
2C \sum_{\stackrel{1 \leq j \leq 2d}{\alpha_j=1}}\|X_j\partial^{\alpha
-e_j}b_d\|_{L^p(\R ^{2d})} +  2C\sum_{\stackrel{1\leq j \leq 2d}{\alpha_{j}
\geq 2}}\alpha_j \|X_j \partial_j \partial^{\alpha-2e_j}b_d\|_{L^p(\R^{2d})}
\\[1ex]
+C \sum _{\stackrel{1\leq j \leq 2d}{2 \leq \alpha_{j} \leq 3}} \alpha_{j}(\alpha_{j}-1)
\|\partial ^{\alpha-2e_j}b_d\| _{L^p(\R ^{2d})}
+ C\sum _{\stackrel{1\leq j \leq 2d}{\alpha_{j} \geq 4}} \alpha_{j}(\alpha_{j}-1)
\|\partial _j^2\partial ^{\alpha-4e_j}b_d\| _{L^p(\R^{2d})}
\\[1ex]
\leq 2C \sum_{\stackrel{1 \leq j \leq 2d}{\alpha_j=1}}J_{\alpha-e_j}+ 2C\sum
_{\stackrel{1\leq j \leq 2d}{\alpha_{j} \geq 2}}\alpha_{j}J_{\alpha-2e_{j}} 
\\[1ex]
+ 6C \sum _{\stackrel{1\leq j \leq 2d}{2 \leq \alpha _{j} \leq 3}} J_{\alpha-2e_j} +
 C\sum _{\stackrel{1\leq j \leq 2d}{\alpha_{j} \geq 4}} \alpha _{j}(\alpha _{j}-1)
J_{\alpha - 4e_{j}}, 
\end{multline}
Using \eqref{reduction} we want now to prove by induction on $|\alpha|\ge 1$ that
\begin{equation} \label{inductiveestimate}
J_{\alpha} \leq C_1^{|\alpha|}(\alpha!)^{1/2}
\end{equation}
for some positive constant $C_1$ depending only on $d$ and on the constant $C=C_p$ 
in \eqref{ellipticestimate}. For $|\alpha|\leq 4, \alpha \neq 0$, the assertion is obvious.
Now assume that it is true for $|\alpha|\leq N-1$ and let us prove it for
$|\alpha|=N.$
We observe that
$$
\alpha_{j}(\alpha-2e_{j})!^{1/2} \leq \sqrt{2}(\alpha_{j}(\alpha_{j}-1)
(\alpha-2e_{j})!)^{1/2} = \sqrt{2}(\alpha !)^{1/2}
$$
and
\begin{multline*}
\alpha_{j}(\alpha_{j}-1) (\alpha-4e_{j})!^{1/2} \leq 2 (\alpha_{j} 
(\alpha_{j}-1)(\alpha_{j}-2) (\alpha_{j}-3)(\alpha-4e_{j})!)^{1/2}
\\[1ex]
=2(\alpha!)^{1/2},
\end{multline*}
Then, from \eqref{reduction} and from the inductive assumption
we obtain
\begin{eqnarray*}
J_{\alpha} &\leq&  4Cd C_1^{|\alpha|-1}(\alpha!)^{1/2} + 4\sqrt{2} Cd
C_1^{|\alpha|-2} (\alpha!)^{1/2} 
\\[1ex]
&&+ 12Cd C_1^{|\alpha|-2}(\alpha!)^{1/2}  + 
4Cd C_1^{|\alpha|-4}(\alpha!)^{1/2} 
 \leq 
C_1^{|\alpha|}(\alpha!)^{1/2}
\end{eqnarray*}
choosing $C_1$ sufficiently large. In particular, from \eqref{inductiveestimate} we obtain that
$$
\sum_{|\gamma|\le 2}\|X^\gamma \partial_{X}^{\alpha}b_d \|_{L^{p}(\R^{2d})}
\le C_1^{|\alpha|+1}(\alpha!)^{1/2}
$$
holds for every $\alpha \in \N^{2d}, \alpha \neq 0.$
Finally, the estimate \eqref{finalestimateho} follows from standard Sobolev
embedding estimates.
\end{proof}

\par

\begin{proof}[Proof of Proposition \ref{harmonicest2}]
In view of Proposition \ref{harmonicest1}, it is sufficient to prove the
estimate \eqref{finalestimateho2} for $|X| \geq 1$.
First we prove that if $p >2d$, then $b_d$ satisfies the following estimate
\begin{equation} \label{primastima}
\| X^{\beta+\tau}\partial_X^{\alpha}b_d\|_{L^p(\R^{2d})} \leq
C_o^{|\alpha|+1}|\alpha|^{|\alpha|}, \qquad X \in \R^{2d},
\end{equation}
for every $\alpha, \beta, \tau \in \N^{2d}$ with $\alpha \neq 0, |\beta|
<|\alpha|$ and $|\tau| \leq 2$. Namely, setting $M=|\alpha+\beta+\tau|$,
we shall obtain \eqref{primastima} by proving the following estimate 
\begin{equation} \label{secondastima}
\|X^{\beta+\tau}\partial_X^{\alpha}b_d\|_{L^p(\R^{2d})} \leq C^{M+1}M^M
\end{equation}
for some positive constant $C$ independent of $M$. We shall argue by
induction on $M$. For $M \leq 4$, the estimate \eqref{secondastima} holds
true since by the Shubin estimate \eqref{estimateh} we have
$$
|X^{\beta+\tau}\partial_X^{\alpha}b_d(X)|  \leq C_1\langle X
\rangle^{-2-|\alpha|+|\beta|+|\tau|} \leq C_1 \langle X \rangle^{-1} \in L^p(\R^{2d}),
$$
when $p>2d$. 

\par

Now let $M > 4$, assume that \eqref{secondastima} holds for
$|\alpha+\beta+\tau| \leq M-1$ and  we shall prove it for $|\alpha + \beta
+ \tau |=M$. First we write
$$
X^{\beta+\tau}\partial_X^{\alpha}b_d= X^{\beta+\tau-\delta} X ^{\delta}
\partial _X^{\alpha-\gamma}\partial_X^{\gamma}b_d,
$$
where we choose $\gamma, \delta$ such that $\gamma \neq 0,
|\gamma+\delta |=M-2$ and $|\alpha-\gamma|+|\beta+\tau-\delta| =2$.
Then, applying \eqref{ellipticestimate}, we get
\begin{multline} \label{terzastima}
\| X^{\beta+\tau}\partial_X^{\alpha}b_d\|_{L^p(\R^{2d})}  \leq \|
X ^{\beta+\tau-\delta}[X^{\delta}, \partial_X^{\alpha-\gamma}]
\partial_X^{\gamma}b_d\|_{L^p(\R^{2d})}
\\[1ex]
+ \| X^{\beta+\tau-\delta}\partial_X^{\alpha-\gamma} (X^{\delta}
\partial _X^{\gamma}b_d)\|_{L^p(\R^{2d})}  
\\[1ex]
\leq 
 \| X^{\beta+\tau-\delta}[X^{\delta}, \partial_X^{\alpha-\gamma}]
 \partial_X^{\gamma}b_d\|_{L^p(\R ^{2d})}
+ C_p\|H_0(X^{\delta}\partial_X^{\gamma}b_d)\|_{L^p(\R^{2d})} 
\\[1ex]
\leq
\| X^{\beta + \tau -\delta}[X^{\delta}, \partial_X^{\alpha -\gamma}]
\partial _X^{\gamma}b_d\| _{L^p(\R^{2d})}
 + C_p\|[H_0, X^{\delta}\partial_X^{\gamma}]b_d\|_{L^p(\R^{2d})},
\end{multline}
since $X^{\delta}\partial_X^{\gamma}(H_0 b_d)=0$. Here we used
the fact that $H_0b_d=1$ and $\gamma \neq 0$. We now estimate
the two terms in the right-hand side of \eqref{terzastima}.

\par

Concerning the first term we have
$$
X^{\beta+\tau-\delta} [X^{\delta}, \partial _X^{\alpha -\gamma}] \partial
_X^{\gamma}b_d = -\sum_{\stackrel{0 \neq \sigma \leq \alpha -\gamma}
{\sigma \leq \delta}}
{{\alpha -\gamma}\choose {\sigma}} \frac{\delta!}{(\delta -\sigma)!}
X ^{\beta + \tau -\sigma}\partial_X^{\alpha -\sigma}b_d.
$$
We can now apply the inductive assumption observing that 
$$
|\alpha-\sigma|+|\beta+\tau-\sigma|=M-2|\sigma|<M<1
$$
and that $\delta !/(\delta-\sigma)! \leq M^{|\sigma|}$ and we obtain
\begin{multline} \label{3*}
\| X^{\beta +\tau -\delta}[X^{\delta}, \partial _X^{\alpha-\gamma}]\partial
_X^{\gamma}b_d\|_{L^p(\R^{2d})}
\\[1ex]
\leq 
\sum _{\stackrel{0 \neq \sigma \leq \alpha -\gamma}{\sigma \leq \delta}}
{{\alpha -\gamma}\choose {\sigma}} M^{|\sigma |}
C^{M-2|\sigma| +1}(M-2|\sigma |)^{M-2|\sigma |}
\\[1ex]
\leq
\sum _{\stackrel{0 \neq \sigma \leq \alpha -\gamma}{\sigma \leq \delta}}
{{\alpha -\gamma}\choose {\sigma}} C^{M-2|\sigma |+1} M^{M- | \sigma |}
\\[1ex]
\leq 
C^{M-1}M^{M-1}\sum_{0 \neq \sigma \leq \alpha -\gamma
} {{\alpha -\gamma}\choose {\sigma}}
\\[1ex]
\leq 2^{|\alpha-\gamma|} C^{M-1}M^{M-1}
\leq 
\frac {4}{C^2} C^{M+1}M^M \leq \frac{1}{2}C^{M+1}M^M
\end{multline}
if $C$ was chosen larger than $2\sqrt{2}$. 

\par 

In order to estimate the second term in the right-hand side of
\eqref{terzastima}, we observe that  the operator $H_0$ is of the form 
$$
H_0= \sum_{|\rho_1| +|\rho_2| \leq 2} c_{\rho_1 \rho_2} X^{\rho_2}
 \partial_X^{\rho_1}
$$
for some constants $c_{\rho_1 \rho_2} \in \R$. Moreover we have
\begin{multline*} 
[X^{\rho_2} \partial_X^{\rho_1}, X^{\delta}\partial_X^{\gamma}] b_d
\\[1ex]
= \sum_{\stackrel{0
\neq \sigma \leq \rho _1}{\sigma \leq \delta}} c^1_{
\rho _1  \delta \sigma} X^{\delta +
\rho _2-\sigma} \partial _X^{\gamma + \rho _1-\sigma}b_d
-\sum _{\stackrel{0 \neq \sigma \leq \rho _2}{\sigma \leq
\gamma}} c^2_{\rho _2 \gamma  \sigma} X^{\delta +
\rho _2-\sigma} \partial _X^{\gamma + \rho _1-\sigma}b_d,
\end{multline*}
where
$$
c^1_{ \rho _1  \delta \sigma} =
{{\rho _1}\choose {\sigma}} \frac {\delta !
}{(\delta -\sigma) !}, \qquad
c^2_{ \rho _2 \gamma  \sigma} =
{{\rho _2}\choose {\sigma}}\frac {\gamma !
}{(\gamma -\sigma) !}.
$$
The constants $c^1_{ \rho _1 
\delta \sigma}, c^2_{ \rho _2 \gamma  \sigma}$ are bounded from
above by $C_3 M^{|\sigma |}$  for some constant $C_3$.
Therefore,
\begin{equation} \label{S1S2} 
\| [X^{\rho_2} \partial_X^{\rho_1}, X^{\delta}
\partial_X^{\gamma} ] b_d \|_{L^p(\R^{2d})} \leq C_3 (S_1+S_2),
\end{equation}
where
\begin{align*}
S_1&= \sum _{\stackrel{0 \neq \sigma \leq
\rho _1}{\sigma \leq \delta}} M^{|\sigma |} \|
X^{\delta + \rho _2 -\sigma} \partial_X^{\gamma +\rho _1 -
\sigma } b_d \| _{L^p(\R^{2d})}
\intertext{and} 
S_2 &= \sum _{\stackrel{0 \neq
\sigma \leq \rho_2}{\sigma \leq \gamma}}  M^{|\sigma |} \|
X^{\delta + \rho _2 -\sigma} \partial_X^{\gamma +\rho _1 -
\sigma } b_d \| _{L^p(\R^{2d})}.
\end{align*}

\par

In order to estimate $S_1$ and $S_2$ in \eqref{S1S2} we observe that
$|\gamma|+|\delta| =M-2$ and $|\rho_1|+ |\rho_2| \leq 2$ imply
$$
|\gamma +\rho_1 -\sigma| + |\delta +\rho_2-\sigma | \leq M
-2|\sigma|.
$$
Then, from the inductive assumption we obtain
$$
\| X^{\delta +\tilde{\beta }- \sigma }\partial_X^{\gamma  +\rho _1 - \sigma}
b_d \|_{L^p(\R^{2d})} \leq C^{M-2|\sigma | +1} (M- 2|\sigma |)^{M-2|\sigma |},
$$
giving that
\begin{equation}\label{3**}
\|  [X^{\rho_2} \partial_X^{\rho _1}, X^{\delta}
\partial_X^{\gamma} ] b_d \| _{L^p(\R^{2d})} \leq 2C_3 C^{M-1}M^M \sum
_{1\leq |\sigma |\leq 2}M^{-|\sigma |}.
\end{equation}
By combining \eqref{terzastima}, \eqref{3*} and \eqref{3**}, and choosing
$C$ sufficiently large, we get
\begin{multline*}
\|X^{\beta+\tau}\partial_X^{\alpha}b_d\|_{L^p(\R^{2d})} \leq \frac{1}{2}C^{M+1}M^M
\\[1ex]
+ 2C_p C_3 C^{M-1}M^M  \sum _{|\rho_1| +|\rho _2| \leq 2} |c_{\rho _1 \rho _2}|
\sum _{1\leq |\sigma |\leq 2}M^{-|\sigma|} \leq C^{M+1}M^M.
\end{multline*}
This gives \eqref{secondastima}.

\par

Now by estimate \eqref{secondastima}, standard Sobolev embedding
estimates and the fact that $|\alpha |^{|\alpha |} \leq C_d^{|\alpha |}\alpha !$,
we get
\begin{equation}\label{quartastima}
|\partial_X^{\alpha}b_d(X)| \leq C^{|\alpha |+1}\alpha ! \langle X \rangle ^{-|\alpha |},
\qquad X \in \R ^{2d}.
\end{equation}
To obtain \eqref{finalestimateho2} for $|X| \geq 1$, we finally use
the fact that by \eqref{bequation} we have
\begin{equation}\label{fromeq}
b_d(X)= \frac{1}{|X|^2}-\frac{1}{4}\frac{\Delta_X b_d(X)}{|X|^2},
\end{equation}
for $|X| \neq 0.$ Here it follows by induction on $|\alpha |$ that
\begin{equation}\label{quintastima}
\left|\partial _X^{\alpha} \left ( \frac{1}{|X|^2} \right) \right | \leq C^{|\alpha |+1}
\alpha ! \langle X \rangle^{-2-|\alpha |}, \qquad |X| \geq 1.
\end{equation}
Hence, by differentiating \eqref{fromeq} and applying \eqref{quartastima}
and \eqref{quintastima} we obtain \eqref{finalestimateho2} for $|X| \geq 1$.
This concludes the proof. 
\end{proof}

\par

So far we have only considered the Weyl symbol of $H^{-1}$. In the following we
make some remarks on the symbol of $H^{-1}$
with respect to other pseudo-differential calculi. More precisely, let
$t\in \mathbf R$, and let $b_{d,t}$ be the $t$-symbol of $H^{-1}$, i.{\,}e.
$b_{d,t}$ is chosen such that
$\op _t(b_{d,t})=H_d^{-1}$ (cf. \eqref{pseudo}). We have
\begin{align}
b_{d,t}(x,\xi ) &= e^{i\tau \scal {D_\xi }{D_x}} b_d(x,\xi ),\qquad
\tau =t-\frac 12
\label{symboltransfer1}
\intertext{(cf. \cite{Ho1}), and by straight-forward computations we get}
b_{d,t}(x,\xi ) &=
C\tau ^{-d}\iint b_d(x-y,\xi -\eta )e^{-i\scal y\eta /\tau}\, dy d\eta,
\label{symboltransfer2}
\end{align}
where the right-hand side is considered as an oscillatory integral, and
should be interpreted as $b_d*\delta _0 =b_d$ when $\tau =0$ (i.{\,}e.
when $t=1/2$, which is the Weyl case). Here the constant $C$ only depends
on the dimension.

\par

We have now the following extension of Proposition \ref{harmonicest1},
where essentially the condition \eqref{finalestimateho} is replaced by
\begin{equation}
\tag*{(\ref{finalestimateho})$'$}
|\partial ^{\alpha}b_{d,t}(x,\xi)| \le
C^{|\alpha|+1}(\alpha!)^{1/2}\eabs {(x,\xi )}^{-2},
\end{equation}

\renewcommand{\rubrik}{Proposition \ref{harmonicest1}$'$}

\begin{tom}
Let $b_{d,t}$ be the $t$-symbol of the inverse to the
harmonic oscillator on $\rr d$. Then there is a constant $C>0$ such
that \eqref{finalestimateho}$'$ holds for every $\alpha \in \N ^{2d}$.
\end{tom}

\par

\begin{proof}
Since the result is the same as Proposition \ref{harmonicest1} when
$t=1/2$, we may assume that $t\neq 1/2$, or equivalently, that
$\tau \neq 0$. We have
\begin{equation}\label{ExpRel1}
e^{-i\scal y\eta /\tau} = \frac {(1-\Delta _\eta )^{d+1}e^{-i\scal y\eta /\tau}}
{(1+|y|^2/\tau ^2)^{d+1}},
\end{equation}
and using this in \eqref{symboltransfer2}, and integrating by parts, we get
$$
b_{d,t}(x,\xi ) =
C\tau ^{-d}\iint \frac {((1-\Delta _\xi )^{d+1} b_d)(x-y,\xi -\eta )}
{(1+|y|^2/\tau ^2)^{d+1}}e^{-i\scal y\eta /\tau}
\, dy d\eta .
$$
In the same way we have
\begin{equation}\label{ExpRel2}
e^{-i\scal y\eta /\tau} = \frac {(1-\Delta _y )^{d+1}e^{-i\scal y\eta /\tau}}
{(1+|\eta |^2/\tau ^2)^{d+1}},
\end{equation}
and again integrations by parts give
\begin{multline}\label{boundedness}
b_{d,t}(x,\xi )
\\[1ex]
=
C\tau ^{-d}\iint (1-\Delta _y)^{d+1}\left (\frac {((1-\Delta _\xi )
^{d+1} b_d)(x-y,\xi -\eta )} {(1+|y|^2/\tau ^2)^{d+1}}\right ) 
\\[1ex]
\times \frac {e^{-i\scal y\eta /\tau}}{(1+|\eta |^2/\tau ^2)^{d+1}}
\, dy d\eta 
\\[1ex]
=
\sum _{ |\beta | \le 4d} C_\beta u_\beta * \psi _\beta ,
\end{multline}
where $u_\beta = \partial ^\beta b_d$, and
$\psi _\beta (x,\xi )$ are equal to
$$
P_\beta (D)(\eabs {x/\tau}^{-2d-2}\eabs {\xi /\tau}^{-2d-2}),
$$
for some differential operator $P_\beta (D)$ with constant coefficients
of order at most $4d$, and which depend on $\beta$ only.
In particular it follows that for some constant $C>0$ we have
$$
|\psi _\beta (X)|\le C\eabs X^{-2d-2},
$$
for every $\beta$.

\par

The result follows if we prove that for every $\beta $, there is a constant
$C$ such that we have
\begin{equation}\label{EstimateAgain}
|(\partial ^\alpha u_\beta ) * \psi _\beta | \le C^{1-|\alpha|}(\alpha !)^{1/2}
\eabs X^{-2}.
\end{equation}
We have
\begin{multline*}
|((\partial ^\alpha u_\beta ) * \psi _\beta )(X)|
\le (|\partial ^\alpha u_\beta | * |\psi _\beta |)(X)
\\[1ex]
\le C_1\sum _{|\beta |\le 4d}
(|\partial ^{\alpha +\beta } b_d | * \eabs \cdo ^{-2d-2})(X)
\\[1ex]
\le C_2\sum _{|\beta |\le 4d}
(|\partial ^{\alpha +\beta } C^{1+{|\alpha +\beta |}}
(\alpha +\beta )!^{1/2}\eabs \cdo ^{-2}  * \eabs \cdo ^{-2d-2})(X)
\\[1ex]
\le C_3^{1+{|\alpha |}}
\alpha !^{1/2}\sum _{|\beta |\le 4d}
(\eabs \cdo ^{-2}  * \eabs \cdo ^{-2d-2})(X)
\\[1ex]
\le C_4^{1+{|\alpha |}}
\alpha !^{1/2} \eabs X ^{-2},
\end{multline*}
for some constants $C_k$, $k=1,\dots ,4$, which only depend on $\beta$.
This proves the result.
\end{proof}

\par

\begin{rem}\label{Remtsymbol}
The techniques in the preceding proof can also be applied to obtain estimates
for $b_{t,d}$, related to Proposition \ref{harmonicest2}, where
the decay should be similar as
in the estimates \eqref{hoscest1}--\eqref{finalestimateho2}. In such approach, 
one needs to apply the operators
\begin{equation}\label{DeltaOps}
\frac {1-\Delta _\eta }{1+|y|^2/\tau ^2}\quad \text{and}\quad
\frac {1-\Delta _y }{1+|\eta |^2/\tau ^2}
\end{equation}
$|\alpha |+2d+2$ times instead of $2d+2$. More precisely, if
$f_r(x)=\eabs x^{-r}$, $x\in \rr d$, then it follows by straight-forward
computations that
\begin{equation}\label{frconv}
|\partial ^\alpha f_{r_1}*f_{r_2} |
\le C^{r}\alpha ! f_{r+|\alpha|},\quad r=\min (r_1,r_2),
\end{equation}
for some constant $C$ which is independent of $r_1,r_2\in \mathbf R$
and $\alpha \in \nn d$ such that $\max (r_1,r_2)\ge d+1$. Here
it seems not possible to replace $r$ by larger values in the inequality
\eqref{frconv}. Consequently,
if the functions which corresponds to
$\partial ^\alpha u_\beta *\psi _\beta$ in the previous
proofs should be bounded by functions of the form $\eabs X^{-2-|\alpha|}$,
it is required that the operators in \eqref{DeltaOps} are applied the asserted
number of times.

\par

This has also consequences on the final estimate. In fact, in the expressions
which corresponds to \eqref{EstimateAgain}, one obtains one factor
$\alpha !$ from $b_d$, because of \eqref{finalestimateho2}, and one such
factor because of the factorial in \eqref{frconv}. Hence, from 
such computations it follows that $b_{d,t}$ satisfies
$$
|\partial ^{\alpha}b_{d,t}(x,\xi)| \le
C^{|\alpha| +1}\alpha !^2\eabs {(x,\xi )}^{-2-|\alpha |},
$$
for some constant $C$ which is independent of $\alpha$. By
taking the geometric mean-value with the previous proposition we
obtain
\begin{equation}\label{1+3sExp}
|\partial ^{\alpha}b_{d,t}(x,\xi)| \le
C^{|\alpha| +1}\alpha !^{(1+3s)/2}\eabs {(x,\xi )}^{-2-s|\alpha |}.
\end{equation}
Consequently, the obtained estimates for $b_{d,t}$ when
$t\neq 1/2$ and $s>0$, are not so
strong compared to the Weyl symbol $b_d$, when using this method
of approximation, since the factor $\alpha !^{(1+3s)/2}$ in
\eqref{1+3sExp} increases faster compared to the factor
$\alpha !^{(1+s)/2}$ in \eqref{finalestimatehoGen}.

\par

On the other hand, by using more refined methods which also
involve symbolic
calclulus it is here conjectured that \eqref{1+3sExp}
can be improved in such way that the factor
$\alpha !^{(1+3s)/2}$ can be replaced by 
a factor $\alpha !^{s_0}$, for some $s_0$ which is strictly smaller
than $(1+3s)/2$, when $s>0$.
\end{rem}

\par

\begin{rem}\label{RemharmonicestExt}
Let $t \in \mathbf R$, and consider the general harmonic oscillator
\begin{equation}\label{ExtHarmOsc}
H = -\Delta +C_1|x|^2+C_2,
\end{equation}
which can be found in Remark \ref{RemGenHarmOsc}. Here
$C_1,C_2\in \mathbf R$ should satisfy \eqref{C1C2consts}. It
follows that $H$ is an invertible and globally elliptic operator
on $\mathscr S$ and $\mathcal S _s$, and their dual spaces,
for every $s\ge 1/2$ (cf. e.{\,}g. \cite{SiTo,To11}). The
inverse $H^{-1}$ of $H$ is a Weyl operator $\op ^w(b_d)$ or a
$t$-operator $\op _\tau (b_{d,t})$, for some
appropriate smooth functions $b_d$ and $b_{d,t}$ on $\rr {2d}$.

\par

For such choices of $C_1$ and $C_2$, it follows by Theorem 
\ref{harmonicest} and Proposition \ref{harmonicest1}$'$,
and their proofs that these results remain valid after the standard harmonic
oscillator has been replaced by  the operator in \eqref{ExtHarmOsc}.
\end{rem}

\par

\section{Explicit formulas for the inverse of the 
harmonic oscillator} \label{sec3}

\par

In this section we derive some formulas for the symbol $b_d(x,\xi)$ starting
from the equation \eqref{bequation} As before, let $X=(x,\xi )\in \rr {2d}$.
By a slight dilation of the variables in $b_d$, we may reformulate \eqref{bequation}
as an equation of the form
$$
H_{2d}F=G,
$$
where $G$ is constant. Hence \eqref{cddef} holds for some real valued $c_d$, in view of Corollary
\ref{harmonicradialsymmetric}. By straight-forward computations,
\eqref{bXequation} and the radial
property \eqref{cddef}, it follows that $c_d$ satisfies
\begin{equation}\label{cd-ekv1}
-tc_d''(t)-dc_d'(t)+tc_d(t) =1. 
\end{equation}

\par

We see from \eqref{cd-ekv1} that $c_d'(0) =-1/d$. Furthermore, we know that
$|X|^2b_d(X)$
is bounded. This implies that $tc_d(t)$ is bounded. Moreover, by Proposition
\ref{harmonicest1} it follows that $c_d$ is extendable to an entire function.
In particular, it is equal to its power series expansion, i.{\,}e.
\begin{equation}\label{c-eq}
c_d(t) = \sum _{k=0}^\infty a_kt^k,
\end{equation}
for some sequence $\{ a_k\} _{k=0}^\infty$.
By differentiations we have
\begin{align*}
tc_d''(t) &= \sum _{k=1}^\infty (k+1)ka_{k+1}t^k,\quad c_d'(t) =
\sum _{k=0}^\infty (k+1)a_{k+1}t^k,
\\[1ex]
tc_d(t) &= \sum _{k=1}^\infty a_{k-1}t^k .
\end{align*}
By inserting this into \eqref{c-eq} we get
$$
\sum _{k=1}^\infty \Big ( -(k+1)(k+d)a_{k+1} + a_{k-1}\Big )t^k -da_1 =1, 
$$
which gives
\begin{equation}\label{ak-cond}
a_1=-\frac 1d,\qquad a_k = \frac {a_{k-2}}{k(k+d-1)},\quad k\ge 2.
\end{equation}

\par

If $k=2p$ is even, then the latter equation gives
$$
a_{2p} = \frac {a_0}{(2p)!! (2p+d-1)(2p+d-3)\cdots (d+1)},
$$
and if $k=2p+1$ is odd, we get
$$
a_{2p+1} = \frac {a_1}{(2p+1)!! (2p+d)(2p+d-2)\cdots (d+2)}.
$$
This gives
\begin{equation}\label{ak-values}
a_{2p} = \frac {a_0(d-1)!!}{(2p)!! (2p+d-1)!!},\quad a_{2p+1}
= \frac {a_1d!!}{(2p+1)!! (2p+d)!!}.
\end{equation}
Here and in what follows we set $0!!=1$, as usual.

\par

Since $a_1=-1/d$, we get
\begin{multline}\label{c-expression}
c_d(t)
\\[1ex]
= 
\frac{d!!}{d} \left (
\alpha \sum _{p=0}^\infty \frac {t^{2p}}{(2p)!! (2p+d-1)!!}
-  \sum _{p=0}^\infty \frac {t^{2p+1}}{(2p+1)!! (2p+d)!!}
\right ),
\end{multline}
where $\alpha =a_0 (d-1)!!d/d!!$. Since $b$ is bounded together with all its
derivatives, the same is true for $c_d(t)$ when $t\ge 0$. This implies that $\alpha$
is uniquely determined. In fact, the right-hand
side of \eqref{c-expression} is a difference of two sums, which both
increase to infinity faster than any polynomial. Hence there is at most one choice of
$\alpha$ such that $c_d(t)$ is bounded when $t\ge 0$.

\par

We claim that $\alpha$ is independent of $d$ when $d$ stays purely among
the even numbers, or purely among the odd numbers. This means that if
$\alpha =\alpha _d$ in \eqref{c-expression}, then we claim that
$\alpha _{d}=\alpha _{d-2}$ for every $d\ge 3$.

\par

In fact, let $\beta =\alpha _d-\alpha _{d-2}$. Then it follows from
\eqref{c-expression} and straight-forward computations that
$$
(d-1)c_d(t)+tc_d'(t)- (d-2)c_{d-2}(t) = \beta (d-2)!! \sum _{p=0}^\infty
\frac {t^{2p}}{(2p)!! (2p+d-3)!!} .
$$
Here the left-hand side is bounded when $t\ge 0$. Since $(d-2)!!>0$ and
the power series on the right-hand side is unbounded, it follows that
$\beta =0$. This proves the stated invariance, as well as
\begin{equation}\label{cd-ekv2}
(d-1)c_d(t)+tc_d'(t)= (d-2)c_{d-2}(t),\qquad d\ge 3.
\end{equation}

\par

\begin{prop}\label{alphaindep}
Let $c_d$ be such that $c_d(|X|^2)$
is the Weyl symbol of the inverse to the harmonic oscillator on $\rr d$.
Then $c_d$ is given by \eqref{c-expression}, where $\alpha=1$ when
$d$ is even, and $\alpha=\pi/2$ when $d$ is odd. 
\end{prop}

\par

\begin{proof}
We may assume that $d=2$ when considering the case when $d$ is even.
By \eqref{c-expression} we have
\begin{multline}\label{firstc2exp}
c_2(t) = \alpha \sum _{p=0}^\infty \frac {t^{2p}}{(2p+1)!}
-  \sum _{p=0}^\infty \frac {t^{2p+1}}{(2p+2)!} 
\\[1ex]
=
\frac {1-(\cosh (t) -\alpha \sinh (t))}{t}.
\end{multline}
Since $c_2$ should be bounded at infinity, it follows from the last expression that
$\alpha =1$, and the result follows in this case.

\par

Next we consider the case when $d$ is odd, and then we may assume that
$d=1$. Let $F$ be the (one-sided) Laplace transform of $c=c_1$. Then the
Laplace transforms of $tc(t)$, $c'(t)$, $tc''(t)$ and $1$ are
$$
s\mapsto -F'(s),\quad s\mapsto sF(s)-c(0),\quad s\mapsto
-2sF(s)-s^2F'(s)+c(0)
$$
and
$$
s\mapsto \frac 1s,
$$
respectively. Hence by Laplace transformation, the equation \eqref{cd-ekv1} 
becomes
$$
(s^2-1)F'(s)+sF(s) =\frac 1s,
$$
and the general solution of this equation is
$$
F(s) = \frac {\arctan (\sqrt {s^2-1}) + C}{\sqrt {s^2-1}},\quad s>1.
$$

\par

Since $c_1(t)$ is bounded for $t\ge 1$, it follows that $F$ is extendable to an analytic
function on the half plane $\operatorname {Re}(s)>0$. This implies that $C=0$,
and $F(s)$ should be interpreted as
$$
F(s) =\sum _{p=0}^\infty \frac {(-1)^p(s^2-1)^p}{2p+1},\quad \text{when}
\quad 0\le s\le \sqrt 2.
$$
Summing up, it follows that
\begin{equation}\label{Laplc1}
F(s) =
\begin{cases}
\displaystyle{\sum _{p=0}^\infty \frac {(-1)^p(s^2-1)^p}{2p+1},
\quad 0< s\le 1,}
\\[3ex]
\displaystyle{\frac {\arctan (\sqrt {s^2-1})}{\sqrt {s^2-1}},\quad s>1.}
\end{cases}
\end{equation}

\par

Now we get
$$
c_1(0) = \lim _{s\to \infty} sF(s) = \lim _{s\to \infty} s\cdot
\frac {\arctan (\sqrt {s^2-1})}{\sqrt {s^2-1}} = \frac \pi 2,
$$
and the result now follows from these equalities and letting $t=0$ in
\eqref{c-expression}. The proof is complete.
\end{proof}

\par

It follows from Theorem \ref{harmonicest} and Fa{\`a} di Bruno's formula,
and the fact $b_d(X) =c_d (|X|^2)$, that for some constant $C$ we have
\begin{equation}\label{cdderestA}
|c_d^{(k)}(t)| \le C^{1+k}(k!)^{(1+s)/2}(1+t)^{-1-sk},\quad t\ge 0,
\end{equation}
for every $s\in [0,1]$ and $k\in \mathbf N$. We shall now go beside
the main stream for a while and combine this inequality with
\eqref{c-expression} to establish narrow estimates for the special
function
\begin{equation}\label{wndef}
w_n(t)\equiv \sum _{p=0}^\infty \frac {t^{2p+1}}{(2p+1)!!(2p+2n+1)!!} ,
\end{equation}
in terms of the Bessel function
\begin{equation}\label{undef}
u_n(t) \equiv \sum _{p=0}^\infty \frac {t^p}{p!(p+n)!}.
\end{equation}

\par

In fact, if $d=2n+1$ is odd, then Proposition \ref{alphaindep},
\eqref{c-expression} and \eqref{cdderestA} give the following result.

\par

\begin{thm}\label{cdrem}
Let $n\in \mathbf N$, $w_n(t)$ and $u_n(t)$ be as in \eqref{wndef}
and \eqref{undef}. Then
\begin{equation}\label{SpecFunctionEst}
\left | \frac {d^k}{dt^k}
\left (
w_n(t) \, -\,
\frac \pi 2u_n(t^2/4) 
\right )
\right |
\le
C^{1+k}(k!)^{(1+s)/2}(1+t)^{-1-sk},\quad t\ge 0
\end{equation}
for some constant $C>0$ which is independent of $s$, $k$ and $t\ge 0$.
\end{thm}

\par

\begin{rem}\label{cdremrem}
Note that the coefficients in the power series in \eqref{SpecFunctionEst} contain
two factors with odd semi-factorials, which can be formulated by four factors
of factorials. It seems to be difficult to find qualitative estimates in the literature
for special functions which are obtained by such power series
expansion  (cf. e.{\,}g. \cite{Olv}). The estimate \eqref{SpecFunctionEst} might then
shed some light on how such functions can be approximated
in terms of the more well-known Bessel functions.
\end{rem}

\medspace

We now continue with our analysis of $b_d$ when $d=2n$ is even.

\par

From \eqref{firstc2exp} we have
\begin{equation}\label{c2formula}
c_2(t) = \frac {1-e^{-t}}t,
\end{equation}
giving that
\begin{equation}\label{bdim2formula}
b_2(X) =\frac{1-e^{-|X|^2}}{|X|^2},\qquad d=2, 
\end{equation}
which can also be rewritten as \eqref{bdim2}.

\par

Moreover, by differentiating \eqref{cd-ekv2} and using \eqref{cd-ekv1} we
obtain the following recursive formula
\begin{equation}\label{cd-ekv3}
tc_d(t) = (d-2)c_{d-2}'(t)+1.
\end{equation}
For example, by \eqref{c2formula} and \eqref{cd-ekv3}, we get
\begin{equation}\label{c4formula}
c_4(t) = \frac {2(t+1)e^{-t}+t^2-2}{t^3}.
\end{equation}
Hence,
\begin{equation}\label{bdim4formula}
b_4(X) = \frac {2(|X|^2+1)e^{-|X|^2}+|X|^4-2}{|X|^6},\qquad d=4.
\end{equation}

\par

Now we aim to prove a general compact formula for $b_d(X)$ in the even 
dimensional case. To this hand, rather than applying the recursive formula
\eqref{cd-ekv3}, we shall proceed by giving first the asymptotic expansion 
of $c_d(t)$ and $b_d(X)$, for any $d \geq 1,$ in terms of homogeneous functions.
In principle, these computations are included in \cite[Section 25]{Sh}, but 
here we need a more explicit result.

\par

\begin{prop}\label{c-asymptotic}
Let $c_d$ be defined by \eqref{cddef}, and let $h_{d,j}(t)$ be given by
$$
h_{d,0}(t)= t^{-1}
$$
and
$$
h_{d,j}(t)= \left ((-1)^j (2j-1)!! \prod \limits _{l=1}^{j}(d-2l)\right ) t^{-1-2j},
\quad j \geq 1. 
$$
Then for every $N \in \N$, $N \geq 1$ and $ n \in \N$, there exists a
positive constant $C_{n,N}$ such that
\begin{equation}\label{aseq} 
\left| \frac{d^n}{dt^n}\left(c_d(t) - \sum_{j<N}h_{d,j}(t)\right)\right| \le
C_{n,N}t^{-1-2N-n}.
\end{equation}
\end{prop}

\par

To prove the proposition we need some preliminary results. First we
note that by \eqref{hoscest1} and \eqref{cddef}, it follows that for every
$n\ge 0$, there is a constant $C_n$ such that
\begin{equation}\label{cdderest}
|c_d^{(n)}(t)| \le C_nt^{-1-n}.
\end{equation}

\par

\begin{lemma}\label{cj-relation}
Let $h_{d,j}$ be the same as in Proposition \ref{c-asymptotic}.
Then
\begin{equation}
\label{eqcj-relation} 
th_{d,j}(t)=th''_{d,j-1}(t)+dh'_{d,j-1}(t), \qquad j \ge 1.
\end{equation}
\end{lemma}

\par

Lemma \ref{cj-relation} follows by straight-forward computation. The details are
left to the reader.

\par

\begin{lemma}\label{cj-inductive}
Let $c_d$ and $h_{d,j}$ be the same as in Proposition \ref{c-asymptotic}.
If $N \ge 1$, then
\begin{multline}\label{eqcj-inductive}
t\left ( c_d(t) - \sum _{j\le N} h_{d,j} (t)\right )
\\[1ex]
= t \left( c_d^{''}(t) - \sum
_{j\le N-1} h^{''}_{d,j} (t) \big)+d \big( c_d '(t) - \sum _{j\le N-1} h'_{d,j} (t) \right ).
\end{multline}
\end{lemma}

\par

\begin{proof}
We prove the lemma by induction on $N$. As $th_0(t)=1$, from
\eqref{cd-ekv1} we get 
\begin{equation}\label{firststep}
t(c_d(t)-h_{d,0}(t)) = tc_d''(t)+dc_d '(t).
\end{equation}
By Lemma \ref{cj-relation} we obtain
\begin{multline*}
t \big ( c_d(t)-h_{d,0}(t)-h_{d,1}(t) \big ) = tc_d''(t)+dc_d'(t)-th_{d,1}(t)
\\[1ex]
= tc_d''(t)+dc_d'(t)-th''_{d,0}(t)-dh'_{d,0}(t) 
\\[1ex]
= t\big ( c_d''(t) -h''_{d,0}(t))+d (c_d'(t) -h'_{d,0}(t) \big ) ,
\end{multline*}
which gives the assertion for $N=1$.

\par

Assume now that \eqref{eqcj-inductive} is true for some $N$
and let us prove it for $N+1$. By \eqref{eqcj-relation} and by
the inductive assumption we get
\begin{multline*} 
t\big( c_d(t)-\sum_{j\le N+1}h_{d,j} (t)\big) = t \left ( c_d(t)-\sum_{j\le N}
h_{d,j} (t)\right) - th_{d,N+1}(t)  
\\[1ex]
=t \left ( c_d^{''}(t)-\sum_{j\le N-1}h^{''}_{d,j} (t)\right ) +d \left ( c_d'(t) -
\sum_{j\le N-1}h'_{d,j} (t)\right )-th^{''}_{d,N}(t)  - dh'_{d,N}(t) 
\\[1ex]
=t\left ( c_d^{''}(t) - \sum_{j\le N}h^{''}_{d,j} (t)\right ) + d \left ( c_d'(t)
- \sum_{j\le N} h'_{d,j} (t) \right ).
\end{multline*}
This proves the lemma.
\end{proof}

\par

\begin{proof}[Proof of Proposition \ref{c-asymptotic}]
First let $N=1$. By \eqref{firststep} we have
\begin{multline*}
\frac{d^n}{dt^n}\big ( c_d(t)-h_{d,0}(t) \big ) =  \frac{d^n}{dt^n}\left ( c_d^{''}(t) +
d\frac{c_d'(t)}{t} \right )
\\[1ex]
= c_d^{(n+2)}(t)+d \sum_{m\leq n} {n \choose m}(-1)^m \frac{m!} {t^{m+1}}
c_d^{(n-m+1)}(t).
\end{multline*}
Hence \eqref{cdderest} gives
$$
\left | \frac{d^n}{dt^n} \big (c_d(t)-h_{d,0}(t) \big ) \right | \le C_1t^{-3-n} +
2^n dC_2 n!t^{-3-n} \leq C_3t^{-3-n},
$$
for some constants $C_1,C_2,C_3$, and \eqref{aseq} follows for
$N=1$. For $N >1$ we argue by induction using Lemma \ref{cj-inductive}.
By \eqref{eqcj-inductive} we have
\begin{multline*}
\frac{d^n}{dt^n} \left ( c_d(t) - \sum _{j\le N}h_{d,j} (t)\right ) = \frac{d^n}{dt^n}
\left (c_d^{''}(t)-\sum_{j\le N-1}h^{''}_{d,j} (t)\right )
\\[1ex]
+ d\sum _{m \le n} {{n}\choose {m}}(-1)^m \frac{m!}{t^{m+1}} \frac{d^{n-m}}{dt^{n-m}}
\left ( c_d'(t)-\sum_{j\le N-1}h'_{d,j} (t)\right ).
\end{multline*}
By the inductive assumption we get
\begin{equation*}
\left | \frac{d^n}{dt^n} \big(c_d(t) -\sum_{j\le N}h_{d,j} (t)\big) \right | 
\leq {C}{t^{-3-2N-n}},
\end{equation*}
for some constant $C$. This gives the result.
\end{proof}

\par

Note that if $d$ is even, we have $h_{d,j}=0$ for $j \geq d/2$. One cannot
expect however that $c_d(t)= \sum\limits_{j =0}^{\infty}h_{d,j}(t)$, since 
the terms in the sum have singularities at the origin. Inspired by \eqref{c2formula},
\eqref{c4formula}, we now define
\begin{equation} \label{h-tildej}
\widetilde {h}_{d,j}(t)=(1-e^{-t}p_{2j}(t))h_{d,j}(t), \qquad j \geq 0,
\end{equation}
where $p_j(t)$ is the Taylor polynomial of $e^t$ of order $j$ centered at $t=0$ and the
functions $h_{d,j}(t)$ are the same as in Proposition \ref{c-asymptotic}. Since terms
with exponential decay do not change the asymptotic expansion, we have for some
positive constants $C_{n,N}$ the following 
\begin{equation}\label{aseq2}
\left | \frac{d^n}{dt^n}\left ( c_d(t) - \sum_{j<N}\widetilde {h}_{d,j}(t)\right )
\right | \le C_{n,N}\eabs t ^{-1-2N-n}
\end{equation}

\par

The singularities at the origin are now cancelled. In the even dimensional case we 
still have $\widetilde{h}_{d,j}(t)=0$ for $j \geq d/2$, and the asymptotic expansion
\eqref{aseq2} becomes indeed an identity as proved below.

\par

\begin{prop}\label{c-asymptoticCorEven}
Let $d=2n>0$ be even, and let $\widetilde {h}_{d,j}$ be
defined by \eqref{h-tildej}. Then
\begin{equation}\label{cdExact}
c_d(t) = \sum_{j=0}^\infty \widetilde {h}_{d,j}(t) = \sum _{j=0}^{n-1}{{n-1}
\choose j}(-1)^j(2j)! \frac {1-e^{-t}p_{2j}(t)}{t^{2j+1}}.
\end{equation}
\end{prop}

\par

\begin{proof}
We shall prove the result by induction.
First we perform some investigations about the sums in \eqref{cdExact}.
We note that $\widetilde {h}_{d,j}=0$ when $j\ge n$, and by straight-forward
computations it follows that the second equality in \eqref{cdExact} hold.

\par

Let $\fy _d(t)$ be the right-hand side of \eqref{cdExact}. By
straight-forward computations we get
$$
\fy _d(t) = \sum _{j=0}^{n-1}{{n-1}\choose j}(-1)^j(2j)! g_j(t),\quad
\text{where}\quad g_j(t) = \frac {1-e^{-t}p_{2j}(t)}{t^{2j+1}},
$$

$$
\widetilde {h}_{d,j}(t) = {{n-1}\choose j}(-1)^j(2j)!g_j(t),\quad \text{and}
\quad tg_j'(t) = -(2j+1)g_j(t) +\frac {e^{-t}}{(2j)!}.
$$
This gives
\begin{multline*}
t\fy '_d(t) +(d-1)\fy_{d} (t)
\\[1ex]
= \sum _{j=0}^{n-1}(2n-2j-2){{n-1}\choose j}(-1)^j(2j)!
g_j(t)
+ e^{-t}\sum _{j=0}^{n-1}{{n-1}\choose j}(-1)^j
\\[1ex]
= 2(n-1)\sum _{j=0}^{n-2}{{n-2}\choose j}(-1)^j(2j)! g_j(t) +
e^{-t}\sum _{j=0}^{n-1}{{n-1}\choose j}(-1)^j.
\end{multline*}

Here the first sum on the right-hand side is $(d-2)\fy _{d-2}(t)$, and the
second sum is zero, by the binomial theorem.

\par

Hence
$$
t\fy '_d(t) +(d-1)\fy_{d}(t) = (d-2)\fy _{d-2}(t),
$$
that is, the sequence $\{ \fy _{2n} \} _{n\ge 1}$ satisfies the same type of
differential equations as $\{ c_{2n} \} _{n\ge 1}$ (cf. \eqref{cd-ekv2}).
In particular, if $\psi _d= c_d-\fy _d$, then $\{ \psi _{2n} \} _{n\ge 1}$ also
fulfills \eqref{cd-ekv2}, after $c_d$ and $c_{d-2}$ have been replaced by
$\psi _d$ and $\psi _{d-2}$, respectively.

\medspace

We now turn into the induction step (over $n$). By the definitions, the result
follows if we prove that $\psi _d=0$ for every $d=2n$. The
result is true for $n=1$, in view of \eqref{c2formula} and by the definition of
$\fy _2$.

\par

Assume that the result is true for $n-1$, i.{\,}e. $\psi _{2n-2}=0$. Then
\eqref{cd-ekv2} implies that $t\psi _{2n}'+(2n-1)\psi _{2n}=0$, giving that
$$
\psi _{2n}(t)=Ct^{1-2n},\quad t > 0,
$$
for some constant $C$. Since $\psi _{2n}(t)$ is continuous for all
$t$ and $t^{1-2n}$ is singular at origin, it follows that $C$ must be zero,
i.{\,}e. $\psi _{2n}=0$. The proof is complete.
\end{proof}

\par

Returning now to $b_d(X)$, we may reformulate Proposition \ref{c-asymptotic} as follows.

\par

\begin{thm} \label{asymptotic}
Let $b_{d,j}$, $j=0,1,\dots$, be given by
$$
b_{d,0}(X)= |X|^{-2}
$$
and
$$
b_{d,j}(X)= \Big ({(-1)^j (2j-1)!! \prod \limits_{l=1}^{j}(d-2l)}\Big )
{|X|^{-2-4j}}, \quad j \geq 1.
$$
Then, for every $N \in \N, N \geq 1$ and for every $\alpha\in \N^{2d}$
the following estimate holds:
\begin{equation}
\Big | \partial_X^\alpha \Big ( b_d(X) -\sum_{j<N}b_{d,j}(X) \Big  )\Big | \le
C_{\alpha, N}|X|^{-2-4N-|\alpha|}
\end{equation}
for some positive constant  $C_{\alpha, N}$ depending only on
$\alpha, N$ and on the dimension $d$.
\end{thm}

\par

Finally, by Proposition \ref{c-asymptoticCorEven} we get the following
result which gives exact formulas for $b_d$ when $d$ is even.

\begin{thm}\label{bExprEvenDim}
Let $d=2n>0$ be even, and let $p_j$ be the Taylor polynomial of $e^t$ of
order $j$ centered at $t=0$. Then
\begin{equation}\label{bExact}
b_{2n}(X) = \sum _{j=0}^{n-1}{{n-1}\choose {j}}(-1)^j(2j)!
\frac {1-e^{-|X|^2}p_{2j}(|X|^2)}{|X|^{2+4j}}.
\end{equation}
\end{thm}

\par

\end{document}